\documentclass[12pt,a4paper]{article}

\usepackage{inputenc}
\usepackage{amsfonts,amsthm,amsmath,amscd,amssymb}

\usepackage[english]{babel}

\newtheorem{teo}{Theorem}[section]
\newtheorem{prop}[teo]{Proposition}
\newtheorem{lem}[teo]{Lemma}
\newtheorem{coro}[teo]{Corollary}

\theoremstyle{definition}

\newtheorem{rem}[teo]{Remark}
\newtheorem{ejem}[teo]{Example}

\def\b{{\cal B}}
\def\b{{\cal B}}

\def\a{{\cal A}}
\def\k{{\cal K}}

\def\T{\mathbb T}
\def\D{\mathbb D}

\def\R{\mathcal R}
\def\bZ{\mathbb Z}
\def\bC{\mathbb C}
\def\C{\mathbb C}
\def\R{\mathbb R}

\def\bh{{\cal B}(H^2)}

\def\ran{{\rm{Ran\, }}}

\def\hc{H^\infty + C}
\def\noi{\noindent}

\newcommand{\ov}[1]{\overline{#1}}

\begin{document}

\title{\vspace*{0cm}Geometric significance of Toeplitz kernels.}

% 58B20 Riemann, Finsler and other ... (global analysis, infinite dimensional)

% 47B35 Toeplitz operators, etc

% 47A63 Operator inequalities

% 53C22 Geodesics

\date{}
\author{E. Andruchow, E. Chiumiento and G. Larotonda \footnote{Supported by Instituto Argentino de Matem\'atica ``Alberto P. Calder\'on" (CONICET), Universidad Nacional de General Sarmiento, Universidad Nacional de La Plata and Agencia Nacional de Promoci\'on de Ciencia y Tecnolog\'ia (ANPCyT).}}

\maketitle

\begin{abstract}
Let $L^2$ be the  Lebesgue space of square-integrable functions on the unit circle.  We show that the injectivity problem for Toeplitz operators is linked to the  existence of geodesics  in the Grassmann manifold of $L^2$. We also investigate this connection  in the context of  restricted Grassmann manifolds associated to $p$-Schatten ideals and  essentially commuting projections. 
\end{abstract}

\section{Introduction}\label{intro}

Let $L^p$ be the usual Lebesgue spaces of complex-valued functions on the unit circle $\T$. The Grassmann  manifold of $L^2$ is the set of all closed subspaces of $L^2$. This paper studies the relation between geodesics on the Grassmann  manifold of $L^2$  and the injectivity problem for Toeplitz operators.

To explain this relation, let $H^2$ be  the Hardy space of the unit circle. Recall that the injectivity problem for Toeplitz operators consists in looking for those symbols $\varphi \in L^\infty$ such that the Toeplitz operator $T_\varphi$ is injective. We relate it to the problem of finding a geodesic on the Grassmann manifold of $L^2$ which joins two subspaces of the form $\varphi H^2$ and $\psi H^2$, where $\varphi, \psi$ are invertible functions in  $L^\infty$. More precisely, we will prove that such a geodesic exists if and only if the Toeplitz operator $T_{\varphi \psi^{-1}}$ and its adjoint both have trivial kernel. 
Furthermore, we will see that these statements are also equivalent to the existence of a minimizing geodesic joining the given subspaces.

 The Grassmann manifold of an abstract Hilbert space  (i.e. the set consisting of all the closed subspaces) may be identified with the  bounded selfadjoint projections. It is an infinite dimensional  homogeneous space  which can be endowed with a Finsler metric by using the operator norm on each tangent space. Although it is complete with the corresponding rectifiable distance, there are subspaces in the same connected component that cannot be joined by a  geodesic (see e.g. \cite{jmaa}). This means that the Hopf-Rinow theorem fails for this manifold.  Nevertheless, much information of its geodesics and their minimizing properties are known. The first results date back to the works \cite{kov, cpr, pr}; both in the more general framework of selfadjoint projections in $C^*$-algebras. %In particular, they showed that each pair of projections lying at distance lower than $1$ can be joined by a minimizing geodesic. 
More recently, there has been progress about the structure of the geodesics in several Grassmann manifolds defined by imposing additional conditions on the subspaces; see for instance \cite{acd, al, alv}  for restricted Grassmann manifolds and \cite{al2} for the Lagrangian Grassmann manifold. 

In this paper, we turn to a more concrete setting by taking the Hilbert space $L^2$. This allows us to study the interplay between geodesics, 
functional spaces and operator theory.  In contrast to the invertibility problem for Toeplitz operators, little attention has been paid  in the literature to the injectivity problem until recent years.  Except for the works of \cite{clark, slee},  the problem remained untreated until the recent works \cite{MP05, MP10, MiP10} (see also the survey \cite{kernels}). Apart from being an interesting problem in operator theory, in these latter articles there are  relevant applications  to harmonic analysis, complex analysis and mathematical physics. 
%, Sturm-Lioville operators and P\'olya sequences.  %We believe that the characterization of injective Toeplitz operators with dense range by means of the existence geodesics in the Grassmann manifold gives a new interesting geometric insight.

The structure of this paper is as follows. In Section 2 we  give classical results on Hardy spaces, Toeplitz and Hankel operators to make the article reasonably self-contained. In Section 3 we prove the aforementioned relation between geodesics of the Grassmann manifold of $L^2$ and the injectivity problem (Theorem \ref{kernels geod}). Then, this result is used to derive an inequality involving the reduced minimum modulus of Toeplitz operators and the norm of a commutator (Theorem \ref{teo des toeplitz}).  

 In Section 4 we deal with  the compact restricted Grassmannian  (or Sato Grassmanian). This is a well-known  Banach manifold related to KdV equations and loop groups (see \cite{sw1, sw}). We need to consider the following two uniform subalgebras of $L^\infty$,  the continuous functions $C$ and the usual Hardy space $H^\infty$.  We show that a subspace $\varphi H^2$ belongs to the compact restricted Grassmannian if and only if $\varphi$ is an invertible function in the Sarason algebra $\hc$.   This  is the least nontrivial closed subalgebra lying between $H^\infty$ and $L^\infty$; it has also been extensively studied \cite{bs90, douglas, s73}.  The existence of geodesics in the  restricted Grassmannian between two subspaces $\varphi H^2$ and $\psi H^2$, $\varphi, \psi$ invertible functions  in $\hc$, depends only on the index of these functions (Theorem \ref{index geod}).  We also examine when a subspace $\varphi H^2$ can be written as $\varphi H^2=gH^2$, where $g$ is a  continuous unimodular function. These results can be carried out also in the setting of restricted Grassmannians associated to $p$-Schatten ideals by using the notion of Krein algebras defined in \cite{bks07}.

Section 5 focuses on   shift-invariant subspaces of $H^2$. Each shift-invaritant subspace can be expressed as $\varphi H^2$, where $\varphi$ is an inner function. We prove that the canonical factorization of $\varphi$ determines the class  where the subspace $\varphi H^2$ belongs (Theorem \ref{producto blaschke}). Based on the results on the injectivity problem mentioned above, we provide examples showing the existence or non existence of geodesics between shift-invariant subspaces.

% Let $C$ be the algebra of continuous functions on $\T$, and consider  the Hardy space $H^\infty$ as a uniformly closed subalgebra of $L^\infty$.  We show   that the subspace $\varphi H^2$ belongs to the compact restricted Grassmannian $Gr_{res}$ if and only if $\varphi$ is an invertible function in the Sarason algebra $\hc$. The compact restricted Grassmannian (or Sato Grassmanian) is a well-known infinite dimensional Banach manifold related to KdV equations and loop groups (see \cite{sw1, sw}). On the other hand, the Sarason algebra is the least nontrivial closed subalgebra lying between $H^\infty$ and $L^\infty$; it has also been extensively studied \cite{douglas, s73} VER 383 DJVU TREATISE PAR ORIGENES SARASON ALG Y TBIEN PAG 5 ALG OF FUCNTIONS, PAPER DE SARASON Then, we investigate conditions to have $\varphi H^2=\theta H^2$  for some  unimodular function $\theta \in C$.

% The result on the compact restricted Grassmannian and the Sarason algebra has a counterpart relating Grassmannians associated to $p$-Schatten classes and Besov spaces.

%The structure  of this paper is as follows. In Section \ref{prelim} we give classical results on Hardy spaces, Toeplitz and Hankel operators to make the article reasonably self-contained.

\section{Background}\label{prelim}

%The classical results on  functional spaces and operator theory stated without proof in this section can be found in several textbooks, see for instance \cite{bs90, douglas, nikolski,   niko86}.

%\medskip

%\noi \textbf{Hardy spaces.}

For $1\leq p \leq \infty$, $L^p=L^p(\T)$ denotes the usual Lebesgue spaces of functions defined on the unit circle $\T$. The Hardy space $H^p$  ($1\leq p < \infty$) is the space of all analytic functions $f$  on   the disk $\D=\{ \, z \in \bC \, : \, |z|<1 \, \}$ for which
$$
\| f\|_{H^p}:= \left(\sup_{0<r<1}\frac{1}{2\pi} \int_0 ^{2\pi} |f(re^{it})|^p\, dt \right)^{1/p} < \infty.
$$
The space of all  bounded analytic functions on $\D$  with the norm $\|f\|_\infty=\sup_{z \in  \D}| f(z) |$ is the Hardy space $H^\infty$.  Functions in Hardy spaces have non tangencial limits a.e.,  a fact which is used to  isometrically identify these spaces with
$$
H^p=\{ \, f \in L^p \,  : \, \int_0^{2\pi} f(e^{it}) \ov{\chi_n(e^{it})} \, dt=0, \, n <0 \, \}.
$$
Here  $(\chi_k)_{k \in \bZ}$ denotes the  orthonormal basis of $L^2$ given by $\chi_k(e^{it})=e^{ikt}$. We shall mostly use this representation of Hardy spaces as functions defined on $\T$ and deal with the values $p=2, \, \infty$. In particular,  $H^2$ is a closed subspace of the Hilbert space $L^2$ and $H^\infty$ is a closed subalgebra of $L^\infty$. For background and notational purposes, our main references for this paper are the books by Douglas,  Nikol'ski\v{\i} and Pavlovi\'c \cite{niko86,nikolski,douglas,pavlovic}.

\medskip

%\noi \textit{Factorization of functions.}
A function $f \in H^2$ is called \textit{inner} if $|f(e^{it})|=1$ a.e. on $\T$. A function $f \in H^2$ is \textit{outer} if
$
\overline{\text{span}} \{ \, f \chi_n  \,  : \, n \geq 0  \, \}=H^2.
$
For each $f \in H^2$, $f \neq 0$, there exist an inner function $f_{inn}$ and an outer function $f_{out}\in H^2$ such that
$f=f_{inn}f_{out}$. This  is called the \textit{inner-outer factorization}, and it is unique up to a multiplicative constant.

%Outer functions are useful to  give a characterization of the invertible functions in the algebra $H^\infty$: $f$ is invertible in this algebra if and only if $f$ is outer and invertible in $L^\infty$.
The inner function can be further factorized.   For each $a \in \D\setminus \{ \, 0 \,\}$,  a Blaschke factor is given by
$$
b_a(z)=\frac{\overline{a}}{|a|} \frac{a-z}{1 - \overline{a}z}, \, \, \, \, z \in \D.
$$
When $a=0$, set $b_0(z)=z$. A  Blaschke product is a function of the form
$$
b(z)=\prod_{j=1}^n b_{a_j}(z), \, \, \, \, z \in \D,
$$
where $1 \leq n \leq \infty$. In the case where $n=\infty$, the infinite Blaschke product
is convergent on compact subsets of $\D$ if the sequence $\{ \, a_j \, \} \subseteq \D$ satisfies the Blaschke condition, that is,  $\sum_j(1 - |a_j|)< \infty$. A finite or infinite Blaschke product is an inner function with zeros given by $\{ \, a_j \, \}$. We remark that the zero set of a holomorphic function in $\D$ satisfies the Blaschke condition.

Let $\mu$ be  a positive finite measure on $\T$. Suppose in addition that $\mu$ is singular with respect to the Lebesgue measure, and set
$$
s_\mu(z)=\exp\left( - \int_{\T} \frac{\psi + z}{\psi -z} \, d\mu(\psi) \right), \, \, \, \, z \in \D.
$$
It turns out that $s_\mu$ is an inner function and $s_\mu(z) \neq 0$ on $\D$.
% and $s_\mu(z)=0$ on the closed support of $\mu$.
A function of this form is known as a singular inner function.

The \textit{canonical factorization} of a function $f\in H^p$ states that there exists a unique factorization $f=\lambda b s_\mu f_{out}$, where $\lambda \in \mathbb C$, $|\lambda|=1$, $b$ is a Blaschke product associated with the zero set of $f$, $s_\mu$ is a singular inner function and $f_{out}$ is the outer part of $f$.
\medskip

\medskip

%\noi \textit{Sarason algebra.}
Let $C$ denote the algebra of continuous functions on $\T$. The \textit{Sarason algebra} is the following algebraic sum
$$
H^\infty + C =\{ \, f+g  \, : \, f \in H^\infty, \, g \in C  \,   \}.
$$
It is proved that this is indeed a closed subalgebra of $L^\infty$. The harmonic extension $\hat{\varphi}$ to   $\D$ of a function $\varphi \in \hc$ is well-defined, and it plays a fundamental role in
%and we recall that is given by
%$$
%\hat{\varphi}(re^{it})= \sum_{k=-\infty}^\infty   \varphi_k r^{|k|}e^{ikt}
%= \frac{1}{2\pi}\int_0^{2\pi}\varphi(e^{it})k_r(\theta - t)dt,
%\, \, \, \, \, \, r \in [0,1), \, \, \, \, t \in [0,2\pi).
%$$
 the characterization of invertible functions in this algebra. For $\varphi \in \hc$ and $0<r<1$,  set $\varphi_r(e^{it})=\hat{\varphi}(re^{it})$. Then $\varphi$ is invertible in $\hc$ if and only if there exist $\delta, \epsilon >0$ such that $|\varphi_r(e^{it})|\geq \epsilon$ for $1-\delta<r<1$ and $e^{it} \in \T$.

This criterion  allows to define the index of an invertible function in $\hc$. For a  non-vanishing function $\varphi \in C$, let $ind(\varphi)\in\mathbb Z$ be the index (or winding number) of $\varphi$ around $z=0$, which for differentiable $\varphi$ can be computed as
$$
ind(\varphi)=\frac{1}{2\pi i}\oint\frac{\varphi'}{\varphi}=\frac{1}{2\pi}\int_0^{2\pi}\frac{\varphi'(e^{it})}{\varphi(e^{it})}\, e^{it}dt.
$$
%In the case where $\varphi$ is rational function without zeros and poles on $\T$, it can be proved that $ind(\varphi)=z-p$, being
%$z$ and $p$ the number of zeros and  poles of $\varphi$ in the unit disk $\D$, respectively.
For $\varphi$ is invertible in $\hc$, set $ind(\varphi)=\lim_{r \to 1^-} ind (\varphi_r)$. This index is stable by small perturbations and it is an homomorphism of the invertible functions in $\hc$ onto the group of integers. The key property to prove these facts as well as the criterion for invertibility is that the harmonic extension is asymptotically multiplicative in $\hc$.

The largest  $C^*$-algebra of $\hc$ is the set of  \textit{quasicontinuous functions}
$$QC=(\hc) \cap (\overline{\hc})$$
Every unimodular $\theta \in QC$ is invertible in $\hc$. In \cite{s73} Sarason proved that each unimodular  function $\theta \in QC$ of index $n \in \bZ$ can be expressed as
$
\theta=\chi_n e^{i(u + \tilde{v})},
$
where $u,v$ are real functions in $C$ and $\tilde{v}$ stands for the harmonic conjugate of $v$ on $\T$.

% (see  \cite[Corollary 165.53.1]{niko86}).

\begin{rem}\label{index comput}
In the case where $\varphi$ is rational function without zeros and poles on $\T$, it is well known that $ind(\varphi)=z-p$, being
$z$ and $p$ the number of zeros and  poles of $\varphi$ in  $\D$, respectively.
 More interesting, when $\varphi$ is a unimodular function sufficiently regular (for instance if $\varphi$ is of bounded variation), the index of $\varphi$ can be computed using its Fourier coefficients $(\varphi_k)_{k \in \mathbb{Z}}$ as
$$
ind(\varphi)=\sum\limits_{k\in\mathbb Z}  k\, |\varphi_k|^2;
$$
see \cite{bourgain} and the references therein.
\end{rem}

\medskip

\noi \textbf{Operators on Hardy spaces.}
The space of bounded linear operators on a Hilbert  space $H$ to a Hilbert space $L$ is denoted by $\b(H,L)$ or $\b(H)$ if $H=L$. Let $H^2_-=\chi_{-1}\ov{H^2}$ be the orthogonal complement of the Hardy space $H^2$,  and consider the orthogonal projections $P_+$ and $P_-$ onto $H^2$ and $H^2_-$, respectively.
 Three special classes of bounded operators will be used in the sequel. For $\varphi \in L^\infty$, the \textit{multiplication operator} $M_\varphi \in \b(L^2)$, $M_\varphi f =\varphi f$, where $f \in L^2$; the \textit{Toeplitz operator} $T_\varphi \in \bh$, $T_\varphi f=P_+ (\varphi f)$, where $f \in H^2$; and the \textit{Hankel operator} $H_\varphi \in \b(H^2,H^2_-)$, $H_\varphi f=P_-(\varphi f)$, where $f \in H^2$.

\medskip

 Recall that the (unilateral) shift operator is given by $M_{\chi_1}$.
It will be useful to state some well-known results on invariant subspaces of the shift operator.

\begin{teo}\label{BHelson}
Suppose that $E$ is a closed subspace of $L^2$ and $M_{\chi_1}E \subseteq E$.
\begin{itemize}
\item[i)] (Wiener) If $E$ is doubly invariant (i.e. $M_{\chi_1}(E)=E$), then $E=\chi_R L^2$ for a unique measurable subset $R \subseteq \T$, where $\chi_R$ is the characteristic of $R$.
\item[ii)] (Beurling-Helson) If $E$ is singly invariant (i.e. $M_{\chi_1}(E)\neq E$), then $E=\theta H^2$ for a unique up to a constant $\theta \in L^\infty$ with $|\theta|=1$ a.e.
\item[iii)] If $0\neq E \subset H^2$, then $E=\theta H^2$ for some inner function $\theta$.
\end{itemize}
\end{teo}
\noi We will frequently use several properties of  Toeplitz operators.  Among the basic properties we recall that  $\|T_\varphi\|=\|\varphi\|_\infty$,  $T_\varphi ^*=T_{\ov{\varphi}}$ and $T_{\varphi \psi}=T_\varphi T_\psi$ whenever $\psi \in H^\infty$. The following  results will be useful.

\begin{teo} (Coburn's lemma)
If $\varphi  \in L^\infty$, then either $\ker(T_\varphi)=\{ 0 \}$ or $\ker(T_{\varphi}^*)=\{ 0 \}$, unless $\varphi\equiv 0$.
\end{teo}

\begin{teo}\label{same prop toep}
Let $\varphi$ be a function in $L^\infty$. The following hold.
\begin{enumerate}
\item[i)]   $T_\varphi$ is invertible if and only if it is Fredholm and has index zero.
\item[ii)] If $\varphi \in \hc$, then $T_\varphi$ is Fredholm if and only $\varphi$ is invertible in $\hc$. Furthermore, the Fredholm index of $T_\varphi$ satisfies $ind(T_\varphi)=-ind(\varphi)$.
\end{enumerate}
\end{teo}

%The Sarason algebra also characterizes  compact Hankel operators.
%\begin{teo} (Hartman)
%If $\varphi \in L^\infty$, then  $H_\varphi$ is compact if and only if $\varphi \in \hc$.
%\end{teo}

\section{The Grassmann manifold of $L^2$}

 Let $Gr$ be the Grassmann manifold of $L^2$, i.e. the set of all closed subspaces of $L^2$.  Let $P_W$ denote the orthogonal projection onto a closed subspace $W \subset L^2$. In particular, we write  $P_\varphi=P_{\varphi H^2}$,
when $\varphi \in L^\infty$ and $\varphi H^2$ is closed. If we  identify each subspace with its orthogonal projection, then
$$
Gr=\{ \, P_W\, : \, W \text{ is a closed subspace of  } L^2  \, \}.
$$
As an application of Theorem \ref{BHelson},  we  determine when $\varphi H^2$ belongs
to $Gr$.

\begin{lem}\label{fih2 closed}
Let $\varphi$ be a nonzero function in $L^\infty$. Then   $\varphi H^2$ is closed in $L^2$ if and only if $\varphi$ is invertible in $L^\infty$.
\end{lem}
\begin{proof}
Clearly, if the function $\varphi$ is invertible in $L^\infty$, then the subspace $\varphi H^2$ is closed. Conversely, suppose that
$\varphi H^2$ is closed. We proceed by way of contradiction and assume that the function $\varphi$ is not invertible in $L^\infty$. We need to distinguish two cases.

In the first case, we assume that there is a Borel set $S \subset \T$ with positive measure such that $\varphi(e^{it})=0$ for all $e^{it} \in S$. Moreover, we may take $S$ to be a maximal set with this property.  Since $\varphi H^2$ is shift invariant, we further need to consider two cases according to whether $\varphi H^2$ is singly or doubly invariant. If $\varphi H^2$ is singly invariant, there is a function $\theta \in L^\infty$ such that $|\theta|= 1$ and $\varphi H^2=\theta H^2$. Then, there is a function $f \in H^2$ such that $\varphi f=\theta$, which is a contradiction since $\varphi\equiv 0$ in $S$. If $\varphi H^2$ is doubly invariant,  then there a Borel set $R\subset \T$ such that $\varphi H^2=\chi_R L^2$. Therefore, $\varphi f=\chi_R$ for some function $f \in H^2$. Recall that for a nonzero function in $H^2$, the set $\{ \, e^{it} \in \T \, : \, f(e^{it})=0\,  \}$ has measure zero (\cite[Thm. 6.13]{douglas}). Using the maximality of $S$, we find that the sets $S$ and $R^c$ must be equal with the
possible exception of points in a set of measure zero. Since $\varphi \neq 0$, $S^c=R$ has positive measure, and we can pick a proper subset $R_1\subset R$ such that $R\setminus R_1$ has positive measure. Again from the equation $\varphi H^2=\chi_R L^2$, we obtain
a nonzero function in $f \in H^2$ such that $\varphi f=\chi_{R_1}$. This implies that $f\equiv 0$ in $R\setminus R_1$, which contradicts the aforementioned property of functions in $H^2$.

In the second case, we suppose that $\varphi \neq 0$ a.e.. If the shift invariant subspace $\varphi H^2$ is doubly invariant, we have again that $\varphi H^2=\chi_R L^2$ for some Borel set $R\subset \T$. In particular, this gives $\varphi=\chi_R \, g$ for $g \in L^2$, and since $\varphi \neq 0$ a.e., it  follows that $\chi_R=1$. Thus, we get $\varphi H^2=L^2$, which certainly cannot be possible.
%: pick the  function $e^{-it} \notin H^2$ and put $h=\varphi e^{-it}$, so there is a function $f \in H^2$ such that $\varphi f= h$, that is, $f=e^{-it}$.
Next we assume that the subspace $\varphi H^2$ is single invariant. Then there is function $\theta \in L^\infty$ satisfying $|\theta|=1$ a.e. and $\varphi H^2=\theta H^2$. We may rewrite this as $\varphi_1 H^2=H^2$, where $\varphi_1=\overline{\theta}\varphi$. Note that $\varphi_1$ is not invertible in $L^\infty$, and  $\varphi_1  \in H^2$, which gives $\varphi_1 \in H^\infty$. Using this fact and that $\varphi_1 \neq 0$ a.e., the Toeplitz
operator $T_{\varphi_1}$ turns out to be injective. Moreover, $T_{\varphi_1}H^2=\varphi_1 H^2=H^2$ shows that $T_{\varphi_1}$ is invertible, and consequently, $\varphi_1$ must be invertible in $L^\infty$ \cite[Thm. 7.6]{douglas}. This gives a contradiction.
\end{proof}

%\begin{rem}
%Instead of dealing with arbitrary invertible functions  in $L^\infty$, in order to represent a subspace $\varphi H^2 \in Gr$, we can assume that $\varphi$ is unimodular. Indeed,  $\varphi H^2$ is a (closed) singly invariant subspace. Hence there exists a function $\theta \in L^\infty$ with $|\theta|=1$ a.e. such that $\varphi H^2=\theta H^2$.
%\end{rem}

%\medskip
%VER SARASON LEE PG 10 PARA EL REMERK ANTERIOR
%\medskip

%A closed subspace $S$ of a Hilbert space $H$ can  be parametrized by  the orthogonal projection $P_S$ onto $S$, and also by the   symmetry (or reflection)  $\epsilon_S$ with respect to $S$ (i.e. the selfadjoint unitary operator which equals $+1$ in $S$ and $-1$ in $S^\perp$).
Let $\a$ be an abstract $C^*$-algebra. Denote by  $Gr(\a)$  the Grassmann manifold of $\a$, i.e.  the set of all selfadjoint projections in $\a$.
In \cite{pr, cpr}, Corach, Porta and Recht decribed the differential geometry of $Gr(\a)$ in terms of projections and symmetries: one passes from projections to symmetries via the affine map
$$
P\longleftrightarrow \epsilon_P=2P-1.
$$
In \cite{cpr}  a natural reductive structure was introduced in $Gr(\a)$. In particular, geodesics were characterized. In \cite{pr}  it was proved that these geodesics have minimal length, if one measures the length of curves by
$$
L(\alpha)=\int_0^1 \| \dot{\alpha}(t) \| \, dt,
$$
where $\alpha: [0,1] \to Gr(\a)$ is a piecewise $C^1$-curve and $\| \, \cdot \,\|$ is the  norm of $\a$.
This means that the operator norm induces a  Finsler metric on $Gr(\a)$; however, note that this metric is not smooth, nor convex.
Let us summarize these facts in the following remark.

\begin{rem}\label{geometria de subespacios}
The Grassmann manifold $Gr(\a)$ is a complemented submanifold of $\a$. Its tangent space $(T Gr(\a))_P$ at $P$ is given by
$$
(T Gr(\a))_P=\{\, Y=iXP-iPX \, : \, X\in \a, X^*=X \, \},
$$
which consists of selfadjoint operators  which are co-diagonal with respect to $P$ (i.e. $PY P=(I-P)Y(I-P)=0$). Denote by $\a_h$ the space of selfadjoint elements of $\a$.
A natural projection $E_{S}:\a \to (T Gr(\a))_P$ is given by
$$
E_P(X)=\hbox{co-diagonal part of }X=P X(I-P)+(I-P)XP.
$$
This map induces a linear connection in $Gr(\a)$: if $X(t)$ is a tangent field along a curve $\alpha(t)\in Gr(\a)$,
$$
\frac{D X}{d t}=E_\alpha({X}).
$$
The geodesics of $Gr(\a)$ starting at $P$ with velocity $Y$ have the form $\delta(t)=e^{t\tilde{Y}}Pe^{-t\tilde{Y}}$, where
$\tilde{Y}=[Y,P]$ is antihermitian and co-diagonal with respect
to $P$.
%\item
%For any $S\in Gr(H)$, the map
%$$
%\pi_S:\u(H)\to Gr(H), \ \ \pi_S(U)=UP_S U^*,
%$$
%($\u(H)$= unitary group of $H$) whose range is the unitary orbit of $P_S$, which regarded again as subspaces is the orbit $\{U(S): U\in\u(H)\}$ of $S$ under the action of $\u(H)$, is a C$^\infty$ submersion. In particular, it has C$^\infty$
%local cross sections.

Let $P$, $Q$ be two orthogonal projections such that $\|P-Q\|<1$. Then there exists a unique operator $X\in\a_h$, with $\|X\|<\pi/2$, which is co-diagonal with respect
to $P$, such that $Q=e^{iX}Pe^{-iX}.$ The curve
\begin{equation}\label{geo g}
\delta(t)=e^{itX}Pe^{-itX}
\end{equation}
 is the unique  geodesic of $Gr(\a)$ joining $P$ and $Q$
(up to reparametrization). Moreover, this geodesic has minimal length. The exponent $X$  is an analytic function of $P$ and $Q$:
%Indeed, since $X$ is $P_0$-co-diagonal, it anticommutes with the symmetry $\epsilon_=\epsilon_{S_0}$:
%$$
%X\epsilon_0=-\epsilon_0X,
%$$
%and therefore
%$$
%\epsilon_t=2\delta(t)-1=e^{itX}\epsilon_0e^{-itX}=e^{2itX}\epsilon_0,
%$$
%and in particular $\epsilon_1=e^{2iX}\epsilon_0$. Thus
$$
X=-\frac{i}{2}\log(\epsilon_p\epsilon_Q),
$$
which is an analytic logarithm because $\|\epsilon_P\epsilon_Q-1\|=\|\epsilon_P-\epsilon_Q\|=2\|P-Q\|<2.$
\end{rem}
More recently, necessary and sufficient conditions were given for the existence of a geodesic  joining two given orthogonal projections in the Grassmann manifold $Gr(H)$ of a Hilbert space $H$. This includes the case in which $\|P-Q\|=1$.
To briefly describe this result, let us recall that Halmos \cite{halmos} (see also \cite{davis, dix}) proposed to understand the geometric properties of  two orthogonal projections $P$ and $Q$ by considering  the decomposition
$$
(\ran(P)\cap \ker(Q)) \, \oplus \, (\ran(Q)\cap \ker(P)) \,\oplus  \, (\ran(P)\cap \ran(Q)) \, \oplus \, (\ker(P)\cap \ker(Q)) \, \oplus \, H_0,
$$
where $H_0$ is the orthogonal complement of the first four subspaces. The projections are said to be in \textit{generic position} when the first four subspaces are trivial. The first two subspaces may be interpreted as an obstruction to find a geodesic between $P$ and $Q$.

\begin{rem}\label{exp geod}
 It was proved in \cite{jmaa} (see also \cite{monteiro})  that there is a   geodesic (equivalently a minimal geodesic) in $Gr(H)$ joining $P$ and $Q$ if and only if
$$
\dim \ran(P)\cap \ker(Q)= \dim \ran(Q)\cap \ker(P).
$$
If both dimensions  are equal to zero, then
there exists a unique geodesic of minimal length in $Gr(H)$ joining $P$ and $Q$. This geodesic has the same form as in (\ref{geo g}) for a (unique) selfadjoint operator $X$ satisfying $\| X\|\leq \pi/2$. In particular, note that there can be a unique minimizing geodesic even if $\|P-Q\|=1$. If the above dimensions do not coincide, then there are infinitely many geodesics.
\end{rem}

Returning to the study of subspaces of the form $\varphi H^2$, we recall   a well-known argument to reduce the injectivity problem of a Toeplitz operator with a general symbol to another one with unimodular symbol.
%\begin{rem}\label{unimodular kernels}
Suppose that $\varphi$ is an invertible function in $L^\infty$. Then there exists a  function $\theta \in L^\infty$, $|\theta|=1$ a.e., such that
$\varphi H^2=\theta H^2$. This gives a function $f \in H^2$ satisfying $\varphi=\theta f$. Note that $f$ is invertible in $L^\infty$. Since $\theta f H^2=  \varphi H^2= \theta H^2$, it follows that $fH^2=H^2$, and then, $f$ is an outer function. Invertible functions in $H^\infty$ are characterized as outer functions which are invertible in $L^\infty$ (see e.g. \cite[Prop. 7.34]{douglas}). Then, $f$ is an invertible function in $H^\infty$, which clearly implies that the Toeplitz operator $T_f$ is invertible. Since $f\in H^\infty$, it follows that $T_\varphi=T_\theta T_f$. Hence  the kernel of $T_\varphi$ is trivial if and only if the kernel of $T_\theta$ is trivial.
%\end{rem}

As a direct consequence of the above results, we can now relate the injectivity problem for Toeplitz operators with the problem of finding a geodesic between two given subspaces $\varphi H^2$ and $\psi H^2$.

\begin{teo}\label{kernels geod}
Let $\varphi, \psi$ be invertible functions in $L^\infty$. The following are equivalent.
\begin{itemize}
\item[i)] $\ker(T_{\varphi \psi^{-1}})=\ker(T_{\varphi^{-1} \psi})=\{ 0 \}$.
\item[ii)] There is a geodesic  in $Gr$ joining  $P_\varphi$ and $P_\psi$.
\item[iii)] There is unique geodesic of  minimal length  in $Gr$  joining  $P_\varphi$ and $P_\psi$ given by
$$
\delta(t)=e^{itX}P_\varphi e^{-itX}, \, \, \, \, t \in[0,1],
$$
where $X=X_{\varphi, \psi}$ is a uniquely determined selfadjoint operator such that $\|X\|\leq \pi / 2$,  $e^{iX}P_\varphi e^{-iX}=P_\psi$, and
it is co-diagonal with respect to both $P_\varphi$ and $P_\psi$.
\end{itemize}
\end{teo}
\begin{proof}
We can assume without loss of generality that $\varphi, \psi$ are unimodular functions by the argument before the statement of this theorem. %Remark \ref{unimodular kernels}.
Then, note that the restriction of the multiplication operator
$$
M_\psi|_{\ker(T_{\overline{\varphi}\psi})}: \ker(T_{\overline{\varphi}\psi}) \to (\varphi H^2)^\perp \cap \psi H^2,
$$
is an isomorphism. Similarly,  $\ker(T_{\varphi \overline{\psi}})\simeq \varphi H^2 \cap (\psi H^2)^\perp$.
If the kernels of both $T_{\varphi\ov{\psi}}$ and $T_{\ov{\varphi}\psi}$ are trivial, then by Remark  \ref{exp geod} there is a geodesic joining $P_\varphi$ and $P_\psi$.  Conversely, if such a geodesic exists, then $\varphi H^2 \cap (\psi H^2)^\perp$ and $(\varphi H^2)^\perp \cap \psi H^2$ have the same dimension. By Coburn's lemma, this dimension must be zero.  Thus, we have shown that the first and second item are equivalent. The equivalence between the second and third item is  explained in Remark \ref{exp geod}.
\end{proof}

\begin{rem}
There are unimodular functions $\varphi$, $\psi$ such that $\ker(T_{\varphi\ov{\psi}})=\ker(T_{\ov{\varphi} \psi})=\{ 0\}$  and $T_{\varphi\ov{\psi}}$ is not invertible. We exhibit a special class of such functions in Example \ref{Toep iny rgo d no inv}.
\end{rem}

\subsection{On the operator $X_{\varphi, \psi}$}

Let us study in more detail the selfadjoint operator $X=X_{\varphi,\psi}$ linking the subspaces $\varphi H^2$ and $\psi H^2$ in Theorem \ref{kernels geod}. To this effect, we recall the following facts concerning Halmos'  model for two orthogonal projections  $P_0$ and $Q_0$ in generic position acting in a Hilbert space $H$.
Under this assumption, there exists an isometric isomorphism between $H$ and a product space $K\times K$ and a positive operator $Z$ in $K$ with $\|Z\|\le \pi/2$ and $\ker(Z)=\{0\}$. This isomorphism transforms the projections $Q_0$ and $P_0$ into
$$
Q_0=\left(\begin{array}{cc} 1 & 0 \\ 0 & 0 \end{array} \right) \ \hbox{ and }\ \ P_0=\left(\begin{array}{cc}C^2 & CS \\ CS & S^2 \end{array} \right),
$$
where $C=\cos(Z)$ and $S=\sin(Z)$ \cite{halmos}.
The  unique selfadjoint operator $X$ linking these projections is  (see \rm{\cite{jmaa}})
$$
X=\left(\begin{array}{cc} 0 & iZ \\ -iZ & 0 \end{array} \right).
$$
Note that $\|X\|=\|Z\|$.

Let $\sigma(A)$ denote the spectrum of an operator $A$. Recall the definition of reduced minimum modulus $\gamma(A)$ of an operator $A\neq 0$:
\begin{align*}
\gamma(A) & =\inf\{\, \|Af\| \, : \, \|f\|=1,  \, f\in \ker(A)^\perp \, \}\\
& =\inf \, \sigma(|A|)\setminus\{0\}.
\end{align*}

\begin{prop}\label{norm xfisi}
Let $\varphi, \psi$ be unimodular functions in $L^\infty$ such that 
$$
\ker(T_{\varphi \ov{\psi}})=\ker(T_{\ov{\varphi} \psi})=\{ 0 \}.
$$
Then
$$
Z=M_{\varphi}\cos^{-1}\left(|T_{\varphi\ov{\psi}}|\right)M_{\ov{\varphi}}
$$
and in particular
$$
\|X_{\varphi,\psi}\|=\cos^{-1}(\gamma(T_{\varphi\ov{\psi}})).
$$
\end{prop}
\begin{proof}
On the non generic part of $P_\varphi$ and $P_\psi$, the operator $X=X_{\varphi,\psi}$ is trivial. Thus in order to compute its norm we restrict to the generic part, and thus they can be described by Halmos' model,
$$
X=\left( \begin{array}{cc} 0 & iZ \\ -iZ & 0 \end{array} \right).
$$
It is elementary that, if $Q_0$, $P_0$ denote the reductions of $P_\varphi$, $P_\psi$ to the generic parts, then
$$
Q_0P_0Q_0=\left( \begin{array}{cc} C^2 & 0 \\ 0 & 0 \end{array} \right).
$$
Now
$$
C^2=P_\varphi P_\psi P_\varphi=M_\varphi P_+ M_{\ov{\varphi}}M_\psi P_+M_{\ov{\psi}}M_\varphi P_+ M_{\ov{\varphi}}=M_\varphi T_{\varphi\ov{\psi}}^* T_{\varphi\ov{\psi}} M_{\ov{\varphi}}=M_\varphi  |T_{\varphi\ov{\psi}}|^2 M_{\ov{\varphi}}.
$$
Therefore $0\le C=\cos(Z)= M_\varphi  |T_{\varphi\ov{\psi}}| M_{\ov{\varphi}}$, and thus, $Z=M_{\varphi}\cos^{-1}\left(|T_{\varphi\ov{\psi}}|\right)M_{\ov{\varphi}}$. From this formula, it follows that
$$\|X_{\varphi,\psi}\|=\| \cos^{-1}(|T_{\varphi\ov{\psi}}|)\|=\cos^{-1}(\lambda_0),$$
 where
$$
\lambda_0 = \inf \, \sigma(|T_{\varphi\ov{\psi}}|)
= \inf \, \sigma(|T_{\varphi\ov{\psi}}|)\setminus\{0\}
= \gamma(T_{\varphi\ov{\psi}}).
$$
The second equality can be deduced from the assumption that $T_{\varphi\ov{\psi}}$ is injective, which implies that $0$ cannot be an isolated point of $\sigma(|T_{\varphi\ov{\psi}}|)$.
\end{proof}

\begin{ejem}\label{ejemplito}
Consider $\varphi=\chi_1$ and the Blaschke factor
$$
\psi(e^{it})=b_a(e^{it})=\frac{\overline{a}}{|a|} \frac{a-e^{it}}{1 - \overline{a}e^{it}} \, ,
$$
for $0<|a|<1$. Then by direct computation,
$$
\varphi H^2 \cap (\psi H^2)^\perp=(\varphi H^2)^\perp \cap \psi H^2=\{0\} \ , \ \  (\varphi H^2)^\perp \cap (\psi H^2)^\perp=H^2_-
$$
and
$$
(\varphi H^2) \cap \psi H^2= \chi_1 b_a H^2= \chi_1 (\chi_1-a)H^2.
$$
Then the generic part $H_0$ of $\varphi H^2$ and $\psi H^2$ is the two dimensional space $H^2\ominus \chi_1 (\chi_1-a) H^2$.
The reduced projections $Q_0=P_\varphi|_{H_0}$ and $P_0=P_\psi|_{H_0}$  are one dimensional,
$$
\ran(Q_0)=H_0\cap \chi_1 H^2=\left\langle \frac{\chi_1}{1-\ov{a}\chi_1} \right\rangle \ , \ \  \ran(P_0)=H_0 \cap (\chi_1-a)H^2=\left\langle \frac{\chi_1 -a}{1-\ov{a}\chi_1}\right\rangle .
$$
According Halmos' formulas,
$$
Q_0P_0Q_0=\left(\begin{array}{cc} C^2 & 0 \\ 0 & 0 \end{array}\right).
$$
Denote by $f$ and $g$ the normalizations of $\frac{\chi_1}{1-\ov{a}\chi_1} $ and $\frac{\chi_1 -a}{1-\ov{a}\chi_1}$, respectively. As usual, let $f_1 \otimes f_2$ be the rank one operator defined by $f_1 \otimes f_2(h)=<h,f_2>f_1$. Then, we have another expression
$$
Q_0P_0Q_0=(f\otimes f)(g\otimes g)(f\otimes f)=|<f,g>|^2f\otimes f.
$$
 Therefore,
$$
\left(\begin{array}{cc} C & 0 \\ 0 & 0 \end{array}\right)=|<f,g>|f\otimes f.
$$
In this case $C=cos(Z)$ is a positive real number, and thus $Z=\cos^{-1}(|<f,g>|)$. Simple computations show that
$|<f,g>|=(1-|a|^2)^{1/2}$, which gives
$$
Z=\cos^{-1}((1-|a|^2)^{1/2})=\sin^{-1}(|a|).
$$
Then, the part of $X_{\varphi,\psi}$ acting on $H_0$ is
$$
X_{\varphi,\psi}|_{H_0}=\left(\begin{array}{cc} 0 & -i\sin^{-1}(|a|) \\ i\sin^{-1}(|a|) & 0 \end{array}\right).
$$
The restriction of $X_{\varphi,\psi}$ to $H_0^\perp$ is trivial. Thus, $X_{\varphi,\psi}$ has rank two, and
$$
\|X_{\varphi,\psi}\|=\sin^{-1}(|a|).
$$
\end{ejem}

%\begin{ejem}
%EJEMPLO CON PRODUCTO DE BLASCHKE DE GRADO N.
%\end{ejem}

\subsection{Norm inequalities}

The minimality property of the geodesics in the Grassmann manifold may be used to obtain   operator inequalities.

\begin{teo}\label{teo des toeplitz}
Let $\varphi, \psi$ be unimodular functions in $L^\infty$ such that $\ker(T_{\varphi \ov{\psi}})=\ker(T_{\ov{\varphi} \psi})=\{ 0 \}$.
Then
$$
\|M_\theta P_+-P_+M_\theta\|\ge \cos^{-1}(\gamma(T_{\varphi\ov{\psi}})),
$$
for every real argument $\theta \in L^\infty$ of the function $\varphi \ov{\psi}$.
\end{teo}
\begin{proof}
Let $\theta$ be a real function in $L^\infty$ such that $e^{i\theta}=\varphi \ov{\psi}$. Consider the curve
$$
\alpha(t)=M_{e^{it\theta}}P_\varphi M_{e^{-it\theta}}.
$$
Apparently, $\alpha(t)$ is a smooth curve in $Gr$ with $\alpha(0)=P_\varphi$ and
$
\alpha(1)=M_{\ov{\varphi}\psi}P_\varphi M_{\varphi\ov{\psi}}=P_\psi .
$
Then $\alpha(t)$ is longer than the (unique) minimal geodesic which joins $\varphi H^2$ and $\psi H^2$, whose length is $\|X_{\varphi,\psi}\|$.
Note that
$$
\dot{\alpha}(t)=i M_{e^{it\theta}}M_\theta P_\varphi-i P_\varphi M_{\theta}M_{e^{-it\theta}}=i M_{e^{it\theta}}M_{\varphi}(M_{\theta} P_+-P_+ M_\theta)M_{\ov{\varphi}}M_{e^{-it\theta}}\,.
$$
Thus, we find that  $\|\dot{\alpha}(t)\|= \|M_{\theta} P_+-P_+ M_\theta\|$, and using Proposition \ref{norm xfisi}, we obtain
$$
\cos^{-1}(\gamma(T_{\varphi\ov{\psi}}))=\|X_{\varphi,\psi}\|\le L(\alpha)=\int_0^1\|\dot{\alpha}(t)\| dt=\|M_{\theta} P_+-P_+ M_\theta\|.
\qedhere$$
\end{proof}

\begin{rem}
With the same hypothesis and notations as in the above theorem, note that the operator $M_\theta$ is selfadjoint. Therefore the commutator $[M_\theta, P_+]=M_\theta P_+-P_+M_\theta$ is anti-hermitian. Also elementary computations show that
$$
P_+ [M_\theta, P_+]P_+=P_-[M_\theta, P_+]P_-=0,
$$
i.e. $[M_\theta, P_+]$  is co-diagonal with respect to $P_+$. Thus, its norm can be related to the norm of the Hankel operator $H_\theta$ by
$$
\|[M_\theta, P_+]\|=\|P_-M_\theta P_+\|=\|H_\theta\|.
$$
Then, by Nehari's theorem (see for instance \cite{niko86}),
$$
\|[M_\theta, P_+]\|=\inf\{\|\theta-f\|_\infty: f\in H^\infty\}.
$$
Hence,
$$
\|X_{\varphi,\psi}\| \leq  \inf\{\|\theta -f\|_\infty: f\in H^\infty\}.
$$
\end{rem}

Special cases of the above inequality can be rephrased without any mention to complex unimodular functions.
\begin{coro}
Let $\theta$ be a real valued continuous function, then
$$
\|M_\theta P_+ - P_+ M_\theta \|\ge \cos^{-1}(\gamma(T_{e^{i\theta}})).
$$
\end{coro}
\begin{proof}
Put $\varphi=e^{i\theta}$ and $\psi=1$ in Theorem \ref{teo des toeplitz}. Then, note that $\varphi$ is an invertible continuous function with zero index. Hence the operator $T_{\varphi}$ is Fredholm and has index zero, which implies that it is invertible.
\end{proof}

Let $\theta_t$, $t \in [0,1]$,  be a piecewise differentiable path of real valued functions in $C$. Then the curve  $\alpha(t)=M_{e^{i\theta_t}}P_+M_{e^{-i\theta_t}}$ is piecewise differentiable. Similarly as above, its velocity is
$$
\|\dot{\alpha}(t)\|=\|M_{e^{i\theta_t}}[M_{i\dot{\theta}_t},P_+]M_{-e^{i\theta_t}}\|=\|H_{\dot{\theta}_t}\|=\inf\{\|\dot{\theta}_t-f\|_\infty: f\in H^\infty\}.
$$
The last quantity can be regarded as the norm of $[\dot{\theta}_t]$, the class of $\dot{\theta}_t$  in the quotient
$L^\infty/H^\infty$ (which is also the velocity of the curve $[\theta_t]$ in the quotient). Therefore,
$$
L(\alpha)=L_{L^\infty/H^\infty}([\theta_t]).
$$
Note that the curve $\theta_t$  is arbitrary between $\theta_0$ and $\theta_1$. In particular, when $\theta_t$ is a straight line, we have the following:
\begin{coro}
Let $\theta_0, \theta_1$ be  real valued continuous functions, then
$$
\|\theta_0-\theta_1\|_{L^\infty/H^\infty}
%\inf\{\|\theta_0-\theta_1+f\|_\infty: f\in H^\infty\}
\ge \|X_{e^{i\theta_0},e^{i\theta_1}}\|= \cos^{-1}(\gamma(T_{e^{i(\theta_1-\theta_o)}})).
$$
\end{coro}

\section{The action of $\hc$ on $Gr_{res}$}\label{sarason y grass sato}

The space $L^2$ has the orthogonal decomposition $L^2=H^2\oplus H^2_-$, which we now use  to give the following definition.
The \textit{compact restricted Grassmannian} $Gr_{res}$  is the manifold of closed linear subspaces $W\subset L^2$ such that
\begin{itemize}
\item $P_+|_W:W\to H^2\in\b(W,H^2)$
is a Fredholm operator, and
\item $P_-|_W:W\to H_- ^2\in\b(W,H_-^2)
$ is a compact operator.
\end{itemize}
The components of the restricted Grassmannian are parametrized by $k\in\mathbb Z$, where $k$ is the index of the operator $P_+|_W:W\to H^2\in\b(W,H^2)$,
$$
Gr_{res}^k=\{W\in Gr_{res}: ind(P_+|_W:W\to H^2)=k\}.
$$
In particular, since $P_+$ is the identity restricted to $H^2$, $H^2=\ran(P_+)\in Gr_{res}^0$.
%More generally, two subspaces $W, S \in Gr_{res}$  belong to the same connected component if and only if $ind(P_+|_W)=ind(P_+|_S)$.

\begin{lem}\label{cond fi grass res}
Let $\varphi$ be an invertible function in $L^\infty$. Then the following are equivalent.
\begin{enumerate}
\item[i)] $\varphi H^2 \in Gr_{res}$.
\item[ii)] $\varphi$ is an invertible function in $\hc$.
\item[iii)] $\varphi H^2=\theta H^2$ for some $\theta\in QC$, $|\theta|=1$ a.e.
%\item[iv)] $\varphi H^2=g H^2$, where $g \in C$ is non-vanishing.
\end{enumerate}
In this case, $\varphi H^2 \in Gr_{res}^k$, where $k=-ind(\varphi)=-ind(\theta)$.
\end{lem}
\begin{proof}
We first prove $i)\Rightarrow ii)$. We claim that the Hankel operator $H_\varphi:H^2 \to H_-^2$, $H_\varphi f=P_-(\varphi f)$, is compact if and only if $P_-|_{\varphi H^2}:\varphi H^2 \to H_-$ is compact. In fact, note that $H_\varphi f=P_- |_{\varphi H^2}(\varphi f)=P_- |_{\varphi H^2}M_\varphi f$, for all $f \in H^2$. Since $\varphi$ is invertible in $L^\infty$, $M_\varphi:H^2 \to \varphi H^2$ is an invertible operator. Thus,
$$
H_\varphi=(P_-|_{\varphi H^2})(M_\varphi|_{H^2}), \, \, \, \, H_\varphi(M_\varphi|_{H^2})^{-1}=P_-|_{\varphi H^2},
$$
which clearly implies our claim.

Suppose that $\varphi H^2 \in Gr_{res}$. Then, the operator $P_-|_{\varphi H^2}:\varphi H^2 \to H_-^2$ is compact, so we get that
$H_\varphi$ is compact. Hartman's theorem asserts that a Hankel operator $H_\varphi$ is compact if and only if $\varphi \in \hc$ (see e.g. \cite[Thm. 2.2.5]{nikolski}). Thus, it follows that $\varphi \in \hc$. Since $\varphi H^2 \in Gr_{res}$, we also have that $P_+|_{\varphi H^2}:\varphi H^2 \to H^2$ is a Fredholm operator. Note that
$\ran(P_+|_{\varphi H^2})=\ran(T_\varphi)$ and $\ker(P_+ |_{\varphi H^2})=M_\varphi \ker(T_\varphi)$, where $T_\varphi$ is the Toeplitz operator with symbol $\varphi$. Therefore $T_\varphi$ is Fredholm, and thus, $\varphi$ is invertible in $\hc$.
 % (see e.g. \cite[Corollary 7.34]{douglas}).

Now we prove $ii)\Rightarrow i)$. Assume that $\varphi$ is an invertible function in $\hc$. Then, we have that $T_\varphi$ is a Fredholm operator. By the same arguments as in the previous paragraph, we see that $P_+|_{\varphi H^2}:\varphi H^2 \to H^2$ is also a Fredholm operator. On the other hand, $\varphi \in \hc$ is equivalent to $H_\varphi$ compact. Hence
$P_-|_{\varphi H^2}:\varphi H^2 \to H_-$ is compact, and consequently, $\varphi H^2 \in Gr_{res}$.

The implication  $ii)\Rightarrow iii)$ is given by Theorem \ref{BHelson}: if $\varphi\in \hc$, then $\varphi H^2$ is singly  invariant. Therefore exists a (unique up to a multiplicative constant) unimodular function $\theta$ such that $\varphi H^2=\theta H^2$. Now $\theta=\varphi f$ for some $f\in H^2$. Since $\varphi$ is invertible in $L^{\infty}$, then $f \in H^\infty$. Hence, $\theta\in \hc$. Further, by the invertibility of $\varphi$, it clearly follows that $f$ is  invertible in $L^\infty$. Using that $\varphi H^2=\theta H^2 =\varphi f H^2$,  we get  $f H^2=H^2$, and consequently, $f$ is an outer function. Recall that a function in $H^\infty$ is invertible if and only if it is outer and invertible  in $L^\infty$. This gives  $f^{-1}\in H^{\infty}$. Now $\overline{\theta}=\theta^{-1}= \varphi^{-1} \, f^{-1}\in \hc$, which proves that $\theta\in QC$.

To prove the implication $iii)\Rightarrow ii)$, we observe that every unimodular $\theta \in QC$ is invertible in $\hc$. By the equivalence between $i)$ and $ii)$, we get $\varphi H^2=\theta H^2 \in Gr_{res}$, and hence $\varphi$ is invertible in $\hc$.

Suppose that $\varphi H^2 \in Gr_{res}^k$. To prove our claim on the index, we have pointed out   that $\ran(P_+|_{\varphi H^2})=\ran(T_\varphi)$ and $\ker(P_+ |_{\varphi H^2})=M_\varphi \ker(T_\varphi)$, where $M_\varphi$ is invertible. It follows that $k=ind(P_+|_{\varphi H^2})=ind(T_\varphi)=-ind(\varphi)$. Moreover, $ \theta= \varphi f$, and $f$ is invertible in $H^\infty$. Every invertible function in $H^\infty$ has index zero. Hence, $ind(\varphi)=ind(\theta)$.
\end{proof}

Under the identification of  each closed subspace $W\subseteq L^2$ with the orthogonal projection $P_W$, the compact restricted Grassmannian
is given by
\begin{equation}\label{gr res operators}
Gr_{res}=\{ \, P \in \b(L^2) \, : \, P-P_+ \text{ is compact}, \, P=P^2=P^*   \, \}.
\end{equation}
Applying the results mentioned in Remark \ref{geometria de subespacios} for the algebra of compact operators, it follows that the  tangent space $(TGr_{res})_P$ at some point $P \in Gr_{res}$ is given by
$$
(T Gr_{res})_P=\{\, iXP-iPX \, : \,  X^*=X \text{ is compact } \, \}.
$$
Then, using the usual operator norm, we have a Finsler metric to measure the length of curves.

On the other hand, the above presentation of $Gr_{res}$ by means of operators is related to the orthogonal projections of the $C^*$-algebra
%of essentially commuting operators
\begin{equation}\label{ess commuti}
\b_{cc}=\{ \, T \in \b(L^2) \, : \, [T,P_+] \text{ is compact}  \, \}.
\end{equation}
Indeed, this algebra consists on operators with compact co-diagonal entries. Denoting by $\pi$ the projection onto the Calkin algebra,
the restricted Grassmannian coincides with the class of projections $P$ such that
$$
\pi(P)=\begin{pmatrix} 1 & 0 \\ 0 &  0 \end{pmatrix},
$$
where this is a matrix decomposition with respect to $\pi(P_+)$ and $\pi(P_-)$.
Metric aspects of the projections in $\b_{cc}$ for a general Hilbert space $H$ were studied in \cite{acd}. In particular, it was proved that
any pair of projections in the same connected component of $Gr_{res}$ can be joined by a geodesic of minimal length. Combining these facts and the characterization in Lemma \ref{cond fi grass res}, we have the following result.

\begin{teo}\label{index geod}
Let $\varphi, \psi$ be invertible functions in $\hc$. The following are equivalent.
\begin{itemize}
\item[i)] $ind(\varphi)=ind(\psi)$.
\item[ii)] There is a geodesic  in $Gr_{res}$ joining  $P_\varphi$ and $P_\psi$.
\item[iii)] There is unique geodesic of  minimal length  in $Gr_{res}$  joining  $P_\varphi$ and $P_\psi$ given by
$$
\delta(t)=e^{itX}P_\varphi e^{-itX}, \, \, \, \, t \in[0,1],
$$
where $X=X_{\varphi, \psi}$ is a uniquely determined compact selfadjoint operator such that $\|X\| < \pi / 2$,  $e^{iX}P_\varphi e^{-iX}=P_\psi$, and
it is co-diagonal with respect to both $P_\varphi$ and $P_\psi$.
\end{itemize}
\end{teo}
\begin{proof}
We first show the equivalence between $i)$ and $ii)$. Suppose that $ind(\varphi)=ind(\psi)$, so we have that $P_\varphi$ and $P_\psi$ belong to the same connected component of $Gr_{res}$.  According to \cite[Thm. 6.6]{acd} there is a (minimal) geodesic joining these projections.
The converse is obvious by the characterization of the connected components of $Gr_{res}$ in terms of the index of the functions.

Similarly, to prove the equivalence between $i)$ and $iii)$, the only non trivial part is that $i)$ implies $iii)$. If $ind(\varphi)=ind(\psi)$, then $ind(\varphi \psi^{-1})=0$, and consequently, as we state in Theorem \ref{same prop toep},  $T_{\varphi \psi^{-1}}$ is an invertible operator.
Following the same argument as in the proof of Theorem \ref{kernels geod},  but now using Lemma \ref{cond fi grass res}, we can assume that
$\varphi$, $\psi$ are unimodular functions in $QC$. Therefore, $\varphi H^2 \cap (\psi H^2)^\perp \simeq \ker(T_{\varphi \ov{\psi}})=\{ 0 \}$ and
$\psi H^2 \cap (\varphi H^2)^\perp \simeq \ker(T_{\psi \ov{\varphi}})=\{ 0 \}$. Under these conditions, there is a unique geodesic of minimal length joining $P_\varphi$ and $P_\psi$ of the desired form (see \cite[Prop. 6.5, Thm. 6.6]{acd}).
%Conversely, as we recall in Remark \ref{geometria de subespacios}, a geodesic in $G_{res}$ joining at $P_{\varphi}$ and $P_\psi$ has the form $\delta(t)=e^{t\tilde{Y}}P_\varphi e^{-t\tilde{Y}}$, where $e^{\tilde{Y}}P_\varphi e^{\tilde{Y}}=P_\psi$. Here $\tilde{Y}=[Y,P_\varphi]$ is antihermitian in $\b_{cc}$ and $P_\varphi$-codiagonal.
\end{proof}

\begin{rem}
As we have seen in the  proof, the above conditions are now equivalent to the invertibility of $T_{\varphi \psi^{-1}}$. A characterization of invertible Toeplitz operators is well studied, see for instance the  Widom-Devinatz theorem in \cite[Thm. 2.23]{bs90}, and \cite[Section 2, Thm. 5]{HNP} for more related results.
\end{rem}

\subsection{Representation by continuous unimodular functions}

Now we address the following question: when can we take the  quasicontinuous function $\theta$ in Lemma \ref{cond fi grass res}  to be continuous? Note that this function is unique up to a multiplicative constant.

%\begin{rem}
The conditions in Lemma \ref{cond fi grass res} are also equivalent to have $\varphi H^2=g H^2$, where $g \in C$ is non-vanishing.   Indeed, this is easily seen from \cite[Corollary 165.50.1]{niko86}, which asserts that the invertibility of a function $\varphi$ in the algebra $\hc$ is equivalent to the factorization $\varphi=fg$, where $f,f^{-1} \in H^\infty$ and $g, g^{-1} \in C$. In addition, note that $ind(g)=ind(\varphi)$.
However, the function $g$ is not necessary unimodular.
%\end{rem}

 Assuming that the function $\varphi$ is continuous, we establish below a relation between   $\theta$ and $\varphi$.
Given a real valued function $u \in L^2$,  $\tilde{u}$ is the harmonic conjugate on $\T$. Denote by $Lip^\alpha$ the Banach space of    complex-valued functions on $\T$ satisfying  a Lipschitz condition of order $\alpha$ ($0<\alpha \leq 1$). We write $A=H^\infty \cap C$ for the disk algebra.

\begin{prop}
Let $\varphi \in C$ be non-vanishing, $\theta$ denote the quasicontinuous function of Lemma \ref{cond fi grass res}, and set   $u=-\log|\varphi|$, then
$$\theta= \frac{\varphi}{|\varphi|}e^{i \tilde{u}}.$$
 In particular,  $\theta \in C$, whenever $\tilde{u} \in C$. In addition, the following assertions hold.
\begin{enumerate}
 \item[i)] If $\varphi \in Lip^\alpha$ for $0< \alpha < 1$, then $\theta \in Lip^\alpha$.
\item[ii)] If $\varphi \in A$, then $\theta \in A$.
\end{enumerate}
\end{prop}
\begin{proof}
Recalling that $\theta H^2=\varphi H^2$, and by the proof of $ii)\Rightarrow iii)$ in Lemma \ref{cond fi grass res}, one can find an invertible function $f$ in $H^\infty$ such that $\theta=f \varphi$. Since $f$ is an outer function, its harmonic extension admits a representation:
$$
\hat{f}(z)=\lambda \exp\left(\frac{1}{2\pi}\int_0^{2\pi} \frac{e^{it}+z}{e^{it}-z} \log|f(e^{it})| dt\right), \, \, \, \, \, z \in \D,
$$
for some $\lambda \in \T$; see \cite[Thm 3.9.6]{niko86}. We may assume that $\lambda=1$. Note that $\hat{f}=\exp(a+ib)$ where
$$
a(z)=\frac{1}{2\pi}\int_0^{2\pi}Re\left\{\frac{e^{it}+z}{e^{it}-z}\right\}\log|f(e^{it})|dt=\log|\hat{f}(z)|,
$$
since the real part of $(e^{it}+z)(e^{it}-z)^{-1}$ is the Poisson kernel. Since $|f|=1/|\varphi|$ on $\T$,  and $f \in H^\infty$, the following radial limit  $\lim_{r\to 1^-} a(re^{it})=\log|f(e^{it})|=u(e^{it})$ exists a.e. On the other hand,
$$
b(z)=\frac{1}{2\pi}\int_0^{2\pi} Im\left\{\frac{e^{it}+z}{e^{it}-z}\right\}\log|f(e^{it})|dt
$$
is the harmonic conjugate of $a$ on $\D$ (up to a constant). By the Privalov-Plessner theorem \cite[Thm. 6.1.1]{pavlovic}, $\lim_{r \to 1^-} b(re^{it})=\tilde{u}(e^{it})$ a.e. Since $\theta=\varphi f$ and $f=e^ue^{i\tilde{u}}=\frac{1}{|\varphi|}e^{i\tilde{u}}$, we obtain $\theta=\frac{\varphi}{|\varphi|}e^{i\tilde{u}}$.

\medskip

%\noi OBSERVACION: LA PRUEBA ANTERIOR SE PODRIA REDUCIR USANDO ESTO: SI $f$ ES UNA FUNCION INVERSIBLE EN $H^\infty$, ENTONCES  $f=\lambda \exp(u + i \tilde{u})$ PARA $\lambda$ CTE Y $u=\log|f|$. ESTO PARECE SER CONOCIDO AUNQUE NO VI UNA PRUEBA (SE USA SIN DEM EN \cite{s73} PG. 7 PDF Y \cite{HNP} PG. 47 PDF)

\medskip

\noi $i)$ Now we assume that  $\varphi \in Lip^\alpha$. Since $\varphi$ is a non-vanishing continuous function, then  $u=-\log|\varphi| \in Lip^\alpha$. By Privalov's theorem, $\tilde{u} \in Lip^\alpha$ for $\alpha<1$ (see \cite[Thm. 10.1.3]{pavlovic}). Clearly, $\varphi, |\varphi|^{-1} \in Lip^\alpha$, which yields $\theta \in Lip^\alpha$.

\medskip

\noi $ii)$ According to \cite[Section 4.3.8]{nikolski}, the outer part $\varphi_{out}$ of $\varphi$ belongs to $A$. Since $\theta=\varphi f$,
it follows that $|f^{-1}|=|\varphi_{out}|$. Therefore, $\varphi_{out}=\lambda f^{-1}$ for some $\lambda \in \T$. Thus, the inner part of $\varphi$ satisfies $\theta=\lambda \varphi_{inn}$, and thus we obtain $\theta \in A$.
\end{proof}

\begin{ejem}
In contrast to what happens with functions in  $Lip^\alpha$ or $A$, we now show that the class of absolutely continuous functions is not preserved in the above proposition. Let
$$
u(e^{it})=-\sum\limits_{n\ge 2} \frac{\sin(nt)}{n\log(n)}
$$
then $u\in C$; moreover $u$ is absolutely continuous on $\T$ \cite[p.241]{zygmund}. Let $\varphi=e^{-u}$, clearly $\varphi\in C$ is non-vanishing and absolutely continuous on $\T$. Since $u(\T)\subset \mathbb R$, we have $\varphi> 0$ on $\T$, therefore $-\log|\varphi|=u$. Let
$$
v(e^{it})=\sum\limits_{n\ge 2} \frac{\cos(nt)}{n\log(n)},
$$
and note that
$$
f(z)=\sum\limits_{n\ge 2} \frac{i}{n\log(n)}z^n=iv+u
$$
is analytic, therefore $v$ is the harmonic conjugate of $u$. But $v$  is not continuous on $\T$, not even bounded since $\sum\limits_{n\ge 2} \frac{1}{n\log(n)}=+\infty$, therefore  $\theta=e^{iv}$ is not continuous on $\T$.
\end{ejem}

\subsection{$p$-norms}

Let $\k(H,L)$ be the space of compact operators between two Hilbert spaces $H$ and $L$.
Given an operator $T \in \k(H,L)$, we denote by  $( s_n(T) )_{n\geq 1}$ the sequence of its singular values.
The $p$-Schatten class ($1\leq p< \infty$) is defined by
$$
\b_p(H,L)=\left\{  \, T \in \k(H,L) \, : \, \|T\|_p=\left(\sum_{n=1}^\infty s_n(T)^p \right)^{1/p}< \infty \, \right\}.
$$
These are Banach spaces endowed with the norm $\| \, \cdot \, \|_p$. As usual, when $p=\infty$, we set $\b_\infty(H,L)=\k(H,L)$. In particular,
 $\b_p(H,H)=\b_p(H)$ is a bilateral ideal of $\b(H)$.
Using the orthogonal decomposition $L^2=H^2\oplus H^2_-$, and the $p$-Schatten class ($1 \leq p < \infty$), one can introduce the \textit{$p$-restricted Grassmannian} $Gr_{res,p}$  as the manifold of closed linear subspaces $W\subset L^2$ such that
 \begin{itemize}
\item $P_+|_W:W\to H^2\in\b(W,H^2)$
is a Fredholm operator, and
\item $P_-|_W:W\to H_- ^2 \in \b_p(W,H_-^2)$.
\end{itemize}
Its connected components $Gr_{res,p}^k$, $k \in \bZ$, are also described by the index of the projection $P_+|_W:W\to H^2$.
The case $p=2$ was studied in connection with loop groups \cite{ps}; it is an infinite dimensional manifold with remarkable geometric properties \cite{brt, go, tump}. Other values of $1\leq p \leq \infty$, or more generally  restricted Grassmannians associated with symmetrically-normed ideals, were treated in \cite{alv, carey}.

%We now recall the definition of Besov spaces following the exposition in \cite{bks07}.  The moduli of continuity of a function $f \in L^p$ are defined as follows: for $s >0$,
%\begin{align*}
%& \ind^1(f,s)= \sup_{|h|\leq s} \|   f(e^{i(\, \cdot \, + h)}) -  f(e^{i \, \cdot \,})  \|_{L^p} \, ; \\
%&  \ind^2(f,s)= \sup_{|h|\leq s} \|  f(e^{i(\, \cdot \, + h)}) -  2f(e^{i \, \cdot \,}) + f(e^{i(\, \cdot \, - h)}) \|_{L^p} \, .
%\end{align*}
% For $1\leq p < \infty$ and $0 < \alpha \leq 1$,  put
%$$ |f|_{B_p ^\alpha} :=\left\{
%\begin{array}{cl}
%\displaystyle{\left(\int_0^{2\pi} (s^{-\alpha}\ind^1(f,s))^p \, \frac{ds}{s}\right)^{1/p}}  & \text{ if } \, \, 0< \alpha <1 ;\\
%\\
%\displaystyle{\left(\int_0^{2\pi} (s^{-\alpha}\ind^2(f,s))^p \, \frac{ds}{s}\right)^{1/p}}  & \text{ if } \, \,  \alpha =1.
%\end{array}\right.
%$$
%The \textit{Besov space} $B_p^\alpha$ is defined as
%$$
%B_p ^\alpha:= \{  \, f \in L^p \, : \, |f|_{B_p^\alpha} < \infty \, \}\,.
%$$
%It is a Banach space endowed with the norm $\| f \|_{B_p^\alpha}:= \| f \|_{L^p} + |f|_{B_p^\alpha}\, .$

We denote by $B_p^\alpha$ the Besov space, where  $1\leq p < \infty$ and $0 < \alpha \leq 1$. For the definition of these spaces, and  the following results we refer to B\"{o}ttcher, Karlovich and Silbermann \cite{bks07}.
%For another description of Besov spaces in terms of special kernels, see for instance \cite[Appendix 5]{niko86}.
%We shall use the following equivalences of Peller's theorem \cite[Appendix 5, Theorem 1.1  and Appendix 4, 23.]{niko86} on Hankel operators and $p$-Schatten operators ($1\leq p < \infty$): given a function $\varphi \in L^\infty$, the following conditions are equivalent:
%\begin{enumerate}
%\item[i)] $H_\varphi \in B_p(H_+ , H_-)$.
%\item[ii)] $P_- \varphi \in B_p ^{1/p}$.
%\item[iii)] $\varphi \in H^\infty + B_p^{1/p}$.
%\end{enumerate}
In this article, among various generalizations of the classical Krein algebra,  it was introduced the following algebra defined by means of Hankel operators:
$$
K_{p,0}^{1/p,0}=\{ \,  \varphi \in L^\infty \, : \,  H_\varphi \in \b_p(H^2,H_-^2) \,  \},
$$
where $1\leq p \leq \infty$. It turns out to be a Banach algebra under the norm
$$
\| \varphi \|_{K_{p,0}^{1/p,0}}= \|\varphi\|_{L^\infty} + \|H_\varphi \|_p \, .
$$
In  the case $p=\infty$, it simply has the usual operator norm of a compact operator. By Hartman's theorem, $K_{\infty,0}^{1/\infty,0}=H^\infty + C$, and for $1\leq p < \infty$, one has $K_{p,0}^{1/p,0} \subseteq H^\infty +C$. Given a function $\varphi \in L^\infty$ and $1\leq p < \infty$, Peller's theorem states that the Hankel operator $H_\varphi \in \b_p(H^2,H^2_-)$ if and only if $P_-\varphi \in B_p^{1/p}$ (see \cite[Thm. 1.1 Appendix 5]{niko86}).
Then there is an equivalent definition of $K_{p,0}^{1/p,0}$ in terms of functions instead of operators. When $1\leq p < \infty$, it holds
$$
K_{p,0}^{1/p,0}=\{ \, \varphi \in L^\infty  \, : \, P_- \varphi \in B_p^{1/p}  \, \}=L^\infty \cap (H^\infty + B_p ^{1/p}).
$$
%and the norms
%$
%\| \, \cdot  \, \|_{K_{p,0}^{1/p,0}}
%$
%and $\| \, \cdot \, \|_{L^\infty} + \| \, \cdot \, \|_{B_p^{1/p}}$ are equivalent.
Moreover, when $p>1$, a function $\varphi$ is invertible in $K_{p,0}^{1/p,0}$ if and only if is invertible in $H^\infty + C$. %(see \cite[Theorem 1.4]{bks07}).

Using the above stated results and the same arguments of Lemma \ref{cond fi grass res},  the following characterization   can be obtained.

\begin{coro}\label{cond fi grass res p}
Let $\varphi$ be an invertible function in $L^\infty$ and $1 \leq p < \infty$. The following assertions are equivalent:
\begin{enumerate}
\item[i)] $\varphi H^2 \in Gr_{res,p}$.
\item[ii)] $\varphi \in K_{p,0}^{1/p,0}$ and $\varphi$ is invertible in $H^\infty + C$.
\item[iii)] $\varphi H^2=\theta H^2$ for some $\theta\in QC \cap K_{p,0}^{1/p,0}$, $|\theta|=1$ a.e.
\end{enumerate}
In this case, $\varphi H^2 \in Gr_{res}^k$, where $k=-ind(\varphi)=-ind(\theta)$.
\end{coro}

\begin{rem}
For $p>1$,  condition $ii)$ can be replaced by
\begin{enumerate}
\item[$ii\, ')$]  $\varphi$ is an invertible function in $K_{p,0}^{1/p,0}$.
\end{enumerate}
\end{rem}

The description for the compact restricted Grassmannian given in (\ref{gr res operators}) has an analogue for the $p$-restricted Grassmannian
\begin{equation}\nonumber
Gr_{res,p}=\{ \, P \in \b(L^2) \, : \, P-P_+ \in \b_p(L^2) , \, P=P^2=P^*   \, \}.
\end{equation}
The tangent space at $P \in Gr_{res,p}$ can be identified with
$$
(TGr_{res,p})_P=\{\, iXP-iPX \, : \,  X^*=X \in B_p(L^2) \, \} \subseteq \b_p(L^2).
$$
Then, a  natural Finsler metric is defined by using the $p$-norm, which gives the following length functional: for  $\alpha: [0,1] \to Gr_{res,p}$ is a piecewise $C^1$-curve,
$$
L_p(\alpha)=\int_0^1 \| \dot{\alpha}(t) \|_p \, dt.
$$
The geodesics defined in (\ref{geo g}) are also minimal for the $p$-norm (see \cite[Corol. 27]{alv}). Thus, we can use the same ideas of Theorem
\ref{index geod} to prove the following (note that $ind(\varphi)=ind(\psi)$ forces $\|P_\varphi-P_\psi\|<1$ by previous remarks):

\begin{coro}
Let $1\leq p < \infty$, and let $\varphi, \psi$ be functions in $K_{p,0}^{1/p,0}$
which are invertible  in $\hc$. The following are equivalent:
\begin{itemize}
\item[i)] $ind(\varphi)=ind(\psi)$.
\item[ii)] There is a geodesic  in $Gr_{res,p}$ joining  $P_\varphi$ and $P_\psi$.
\item[iii)] There is unique geodesic of  minimal length  in $Gr_{res,p}$  joining  $P_\varphi$ and $P_\psi$ given by
$$
\delta(t)=e^{itX}P_\varphi e^{-itX}, \, \, \, \, t \in[0,1],
$$
where $X=X_{\varphi, \psi}$ is a uniquely determined selfadjoint operator such that $\|X\| < \pi / 2$,  $e^{iX}P_\varphi e^{-iX}=P_\psi$, and
it is co-diagonal with respect to both $P_\varphi$ and $P_\psi$.
\end{itemize}
\end{coro}
Moreover, arguing as in the proof of Theorem \ref{teo des toeplitz} we also obtain
\begin{coro}
Let $1\leq p < \infty$, and let $\varphi, \psi$ be functions in $K_{p,0}^{1/p,0}$
which are invertible  in $\hc$, such that $ind(\varphi)=ind(\psi)$. Then if $\theta\in K_{p,0}^{1/p,0}$ is such that $e^{i\theta}=\varphi \ov{\psi}$,
$$
\|M_\theta P_+-P_+M_\theta\|_p\ge 2^{1/p}\| \cos^{-1}(|T_{\varphi\ov{\psi}}|)\|_p=dist_p(P_\varphi,P_\psi).
$$
For instance, if $\varphi$ and $\psi$ are $C^1$ functions (with equal index) such an argument $\theta$ exists,
which is continuous and piecewise smooth.
\end{coro}
\begin{proof}
 Recall from Poposition \ref{norm xfisi} that
 $$
 X_{\varphi, \psi}=\left( \begin{array}{cc} 0 & iZ \\ -iZ & 0 \end{array}\right)
 $$
 and thus ($Z\ge 0$)
 $$
 |X_{\varphi, \psi}|=\left( \begin{array}{cc} Z & 0 \\ 0 & Z \end{array}\right) .
 $$
 Also $Z=M_\varphi cos^{-1}(|T_{\varphi\ov{\psi}}|) M_{\ov{\varphi}}$. Then
 $$
 \|X_{\varphi, \psi}\|_p=2^{1/p}\|Z\|_p=2^{1/p}\|cos^{-1}(|T_{\varphi\ov{\psi}}|)\|_p
 \qedhere $$
\end{proof}

\section{Shift-invariant subspaces of $H^2$}

%Now we study  the problem of joining by a minimal geodesic two subspaces of the form $\varphi H^2$ and $\psi H^2$, where $\varphi$ and $\psi$ are inner functions.
%It turns out that these subspaces belong to either the restricted Grassmannian or the essential class $\mathbb{E}_1$.

The orthogonal projections of the $C^*$-algebra $\b_{cc}$ defined in (\ref{ess commuti}) may be classified using their
image in the Calkin algebra. In addition to the restricted Grassmannian, we shall need to consider the essential class  $\mathbb{E}_1$ consisting of all the orthogonal projections which have the form
(in terms of $\pi(P_+)$ and $\pi(P_-)$)
$$
\pi(P)=\begin{pmatrix} p & 0 \\ 0 &  0 \end{pmatrix},
$$
where $p\neq 0,1$ is a projection in the Calkin algebra.  It was shown that the class $\mathbb{E}_1$ is connected, and in contrast to the restricted Grassmannian, there are projections which cannot be joined by a geodesic in $\mathbb{E}_1$.
% (see \cite[Thm. 5.3]{acd}).

Let $E$ be a closed subspace of $L^2$ such that $M_{\chi_1}(E)\subset E$. If $0\neq E \subseteq H^2$, then $E=\varphi H^2$ for some inner function $\varphi$. We prove below that these subspaces belong to either the restricted Grassmannian or the essential class $\mathbb{E}_1$.

\begin{teo}\label{producto blaschke}
 Let $\varphi$ be an inner function. Then the following assertions hold:
\begin{enumerate}
\item[i)] $\varphi$ is  a finite Blaschke product if and only if $P_\varphi \in Gr_{res}^k$, where $k$ is the number of zeros  of $\varphi$.
\item[ii)] $\varphi$ is not  a finite Blaschke product if and only if $P_\varphi \in \mathbb{E}_1$.
\end{enumerate}
\end{teo}
\begin{proof}
$i)$ The only inner functions which are invertible in $\hc$ are the finite Blaschke products  (see e.g. \cite[Thm. 5]{s32}). Therefore, the result follows  from Lemma \ref{cond fi grass res}. The index of a Blaschke factor is equal to its number of zeros (see Remark \ref{index comput}), and as we have already showed, it determines the connected component of $Gr_{res}$ where $P_\varphi$ lies.

\medskip

\noi $ii)$ Suppose that $\varphi$ is not a finite Blaschke product. As we remarked in the preceding item, this means that $\varphi$ is not invertible in $\hc$. Therefore, $P_\varphi \notin Gr_{res}$ by Lemma \ref{cond fi grass res}.
On the other hand, by the claim proved in the first paragraph of the same lemma, we know that $P_- P_\varphi |_{\varphi H^2 }: \varphi H^2 \to H_- ^2$ is compact, since $\varphi \in H^\infty$. Hence $a^*=P_-P_\varphi |_{H^2}$ is also compact, so that $P_\varphi \in {\b}_{cc}$. Similarly, we also find that $y=P_-P_\varphi|_{H_- ^2}$ is compact. Now recall that a projection
$$
P=\left( \begin{array}{ll} x & a \\ a ^*&  y \end{array} \right),
$$
belongs to $G_{res}$ if and only if $a,y$ are compact operators and $x$ is Fredholm (see \cite[Lemma 3.3]{acd}). Applying this to $P=P_\varphi$  we obtain that $x=P_+P_\varphi |_{H^2}$ is not Fredholm. In order to prove that $P_\varphi \in \mathbb{E}_1$, it only remains to verify that $x$ is not compact. To this end, it suffices to show that $\dim \ker(x-1)=\infty$.  But since $\varphi \in H^\infty$, we have
$\ker(x- 1)=H^2 \cap \varphi H^2 =\varphi H^2$, which  has infinite dimension. The converse is an immediate consequence of Lemma \ref{cond fi grass res} and the characterization of invertible inner functions in $\hc$.
\end{proof}

%Note that every geodesic in $\mathbb{E}_1$ must be a geodesic in $Gr$.
\begin{rem}
 Every geodesic in $\mathbb{E}_1$ is a geodesic in $Gr$. This follows by the explicit form of geodesics in a general $C^*$-algebra described in Remark \ref{geometria de subespacios}. However, the converse does not hold: geodesics in $Gr$ joining two projections of $\mathbb{E}_1$ may lie outside of $\mathbb{E}_1$. Suppose that $P, Q \in \mathbb{E}_1$, and there is a geodesic $\delta(t)=e^{itX}Pe^{-itX}$ in $Gr$ joining these projections. Then $\delta(t)$ belongs to $\mathbb{E}_1$ if and only if $X \in \b_{cc}$  (see \cite[Prop 6.11]{acd} for other equivalent conditions).
\end{rem}

%\medskip

%OBSERVACION: HABIA UN EJEMPLO DE GEODESICA EN $Gr$ QUE NO LO ERA EN $\mathbb{E}_5$. SIN EMBARGO NUNCA PENSAMOS SI SE PODIA HACER UN EJEMPLO ASI EN $\mathbb{E}_1$.

\subsection{Examples}

We shall give examples of shift-invariant subspaces which can or cannot be joined by a (minimal) geodesic in the Grassmann manifold $Gr$.
The simplest case  is a consequence
of the following result  proved in \cite[Lemma 3.2]{MP05} for Hardy spaces of the upper half-plane.
It is an elementary but important step to understand Toeplitz kernels. We shall state it for the Hardy space of the  circle.

\begin{lem}\label{basic criterion}
Let $\varphi, \psi$ be two inner functions. Then $\ker(T_{\varphi \overline{\psi}})\neq \{ 0 \}$ if and only if
there exist an inner function $\theta$ and an outer function $g$ such that
$
 \varphi \, \theta  \, g=\psi \overline{g}
$ on $\T$.
\end{lem}

\begin{ejem}
Suppose that $ \varphi$ divides $\psi$. This means that there is an inner function $\theta$ such that
$\varphi \theta=\psi.$
 Thus, the equation in Lemma \ref{basic criterion} is satisfied with $g=1$, and  consequently, $\ker(T_{\varphi\overline{\psi}})\neq\{ \, 0 \, \}$. Hence there is no  geodesic  in $Gr$ joining $\varphi H^2$ and $\psi H^2$. Note that $\ker(T_{\ov{\varphi}\psi})=\{ \, 0 \, \}$.
In this case, it is not difficult to construct concrete examples using the following well-known description of divisors in $H^\infty$. Suppose that $\{ a_j \}$ and $\{  a_j ' \}$ are the zero sets of $\varphi$ and $\psi$, respectively. If $\varphi=\lambda b s_\mu $ and $\psi= \lambda' b' s_{\mu'}$ are the canonical factorizations, then
$\varphi$ divides $\psi$ if and only if  $\{ a_j \} \subseteq \{  a_j ' \}$ and $\mu \leq \mu'$.
\end{ejem}

 The canonical factorization factorization also turns out to be relevant  to give an affirmative answer to the existence of a geodesic in many concrete cases. Let $\varphi$ be an inner function. A point  on $\T$ belongs to the \textit{support} of $\varphi$ if it is a limit point of zeros  of $\varphi$ or if it belongs to the support of the singular measure associated with the singular factor of $\varphi$. We write $supp(\varphi)$ for the support of $\varphi$. Sarason and Lee proved the following \cite[Thm. 1-2]{slee}.

\begin{teo}
Let $\varphi$, $\psi$ be inner functions.
\begin{enumerate}
\item[i)] If $supp(\varphi)\neq supp(\psi)$, then the spectrum of $T_{\varphi \ov{\psi}}$ is the closed unit disk.
\item[ii)] If there is a point $z_0 \in supp(\psi) \setminus supp(\varphi)$, then $T_{ \varphi \ov{\psi}}- \lambda$ has dense range for all $\lambda$.
\end{enumerate}
\end{teo}

\noi From the above result and Theorem \ref{kernels geod} we obtain this example.

 \begin{ejem}\label{Toep iny rgo d no inv}
 Let $\varphi$, $\psi$ be inner functions. Suppose that there are two points $z_0$ and $z_1$ such that  $z_0 \in supp(\psi) \setminus supp(\varphi)$ and  $z_1 \in supp(\varphi) \setminus supp(\psi)$. Then there is unique minimal geodesic in $Gr$ joining $P_\varphi$ and $P_\psi$ of the form stated in  Theorem \ref{kernels geod}.
\end{ejem}

Now we consider the case of two inner functions with support $z=1$. As a direct consequence of the results on Toeplitz kernels obtained by Makarov, Mitkovski and  Poltoratski \cite{MP05, MiP10}  (see also the survey \cite{kernels}), one can show examples of the two inner functions of the aforementioned type such that their corresponding subspaces can or cannot be joined by a geodesic in $Gr$. These remarkable results were proved for Toeplitz operators in Hardy spaces of the upper-half plane (and other classes of functions). For this reason, we shall change to the half-plane; however
by the isometry exhibited below all can be translated to the disk.

\medskip

A function $F$ holomorphic on the upper half-plane $\mathbb{C}_+=\{ \, z \, : Im z >0 \, \}$ belongs to the
Hardy  space $H^2_+=H^2(\C_+)$ if
$$
\|F\|_{H^2_+}:=\left(\sup_{y>0}\int_{-\infty} ^\infty |F(x+iy)|^2 dx \right)^{1/2} <\infty.
$$
As in Hardy spaces of the disk, one may consider $H^2_+$ as a Hilbert subspace of $L^2(\R)$ since non tangencial limits exist a.e. No confusion will arise if we also denote by $P_+$  the orthogonal projection of $L^2(\R)$ onto $H^2_+$. The Toeplitz operator with symbol $U \in L^\infty(\R)$ is defined by
$$
T_U:H^2_+ \to H^2_+, \, \, \, \, T_U(F):=P_+(UF).
$$
We write $H^\infty_+=H^\infty(\C_+)$ for the bounded holomorphic functions on $\C_+$. Notice that $w=\frac{z-i}{z+i}$
is a conformal map from $\C_+$ onto $\D$. Set $f(w)=F(z)$. Then, it follows that $F(z) \in H^\infty_+$ if and only if
$f(w) \in H^\infty$. However, $H^2_+$ is not obtained from $H^2$ by conformal mapping. It can be shown
that $f(w) \in H^2 $ if and only if $\frac{\pi^{-1/2}}{(z+i)} F(z) \in H^2_+$. Taking  boundary values, one sees that
$$
W:H^2 \to H^2_+, \, \, \, \, \, Wf(x)=\frac{\pi^{-1/2}}{(x+i)} \, f\left(\frac{x-i}{x+i}\right), \, \, \, x \in \R.
$$
is an isometry from $H^2$ onto $H^2_+$. Set $\gamma(x)=\frac{x-i}{x+i}$ and fix $\theta \in L^\infty$. Then,   Toeplitz operators in the Hardy spaces of the disk and the upper half-plane are related by
$$
W T_\theta = T_{\theta  \circ \gamma} W\, .
$$
The canonical factorization of functions in $H^2$ can be also derived in $H^2_+$ using the isometry $W$.

\medskip

By an inner function $\Theta$ in $\C_+$ we mean that $\Theta \in H^\infty_+$ and $|\Theta|=1$ on $\R$. An inner function $\Theta(z)$ in $\C_+$ is  a \textit{meromorphic inner function} if it has a meromorphic extension to $\C$. In this case, the meromorphic extension to the lower half-plane is given by $\Theta(z)=\frac{1}{\overline{\Theta(\overline{z})}}$.
Each meromorphic inner function $\Theta$ admits a canonical  factorization
$\Theta=B_\Lambda S^a$, where $a\geq 0$ and $\Lambda$ is a discrete set in $\C_+$ without accumulation points on $\R$ such that the following Blaschke condition holds
$$
\sum_{\lambda \in \Lambda} \frac{Im \lambda}{1+|\lambda|^2} < \infty.
$$
The function $B_\Lambda$ is the corresponding Blaschke product in $\C_+$, i.e.
$$
B_\Lambda (z)= \prod_{\lambda \in \Lambda} \epsilon_\lambda \frac{z-\lambda}{z-\overline{\lambda}}; \, \, \, |\epsilon_\lambda|=1.
$$
The other function in the factorization is given by the singular inner function
$S^a(z)=e^{iaz}$.  Meromorphic inner functions correspond to inner functions in $H^2$ such that $z=1$ is the only possible accumulation point of their zeros and also the only possible singular point mass.
%Further,  a meromorphic inner function $\Theta$ has a representation  $\Theta=e^{i\theta}$ on $\R$, where $\theta$ is an   real analytic increasing function.

\begin{ejem}
The \textit{point spectrum} of a meromorphic inner function $\Theta=B_\Lambda S^a$ is the set $\sigma(\Theta)=\{ \, \Theta=1 \,  \}$ or $\{ \, \Theta=1 \,  \}\cup \{ \infty \}$. The point $\infty$ belongs to the spectrum if $\sum_{\lambda \in \Lambda} Im \lambda < \infty$ and $S^a\equiv 1$ (see \cite{MP05} for other equivalent conditions).  Two meromorphic inner functions are said to be \textit{twins} if they have the same point spectrum, possibly including infinity.
The twin inner function theorem asserts that if $\Theta$, $J$ are twins, then $\ker(T_{\ov{\Theta} J})=\{0\}$ \cite[Thm. 3.20]{MP05}. Thus, there is always a geodesic joining the corresponding subspaces defined by twin functions.
\end{ejem}

\begin{ejem}
 Recall that a sequence of real numbers  is separated if $|\lambda_n - \lambda_m|\geq \delta >0$ $(n \neq m)$.
A separated sequence $( \lambda_n)_{n \in \mathbb{Z}}$ is a called a P\'olya sequence if every zero-type entire function bounded on $( \lambda_n)_{n \in \mathbb{Z}}$  is constant (see also \cite{MiP10} for a new characterization). Among several  conditions, it was proved in \cite[Thm. A]{MiP10} that $( \lambda_n  )_{n \in \mathbb{Z}}$ is a P\'olya sequence if and only if there exists a meromorphic inner function $\Theta$  with $\{  \, \Theta=1\,\}=( \lambda_n  )_{n \in \mathbb{Z}}$ such that $\ker(T_{\overline{\Theta}S^{2c}})\neq \{ 0 \}$ for some $c>0$. Hence there is no geodesic joining the corresponding subspaces defined by $\Theta$ and $S^{2c}$.
\end{ejem}

\section*{Acknowledgements} This research was supported by CONICET (PIP 2016 112201) and ANPCyT (2010 2478). We would like to thank Daniel Su\'arez for his helpful insight on Toeplitz operators.

{\sc (Esteban Andruchow)} {Instituto de Ciencias,  Universidad Nacional de Gral. Sar\-miento,
J.M. Gutierrez 1150,  (1613) Los Polvorines, Argentina and Instituto Argentino de Matem\'atica, `Alberto P. Calder\'on', CONICET, Saavedra 15 3er. piso,
(1083) Buenos Aires, Argentina.}

\noi e-mail: {\sf eandruch@ungs.edu.ar}

\bigskip

{\sc (Eduardo Chiumiento)} {Departamento de de Matem\'atica, FCE-UNLP, Calles 50 y 115, 
(1900) La Plata, Argentina  and Instituto Argentino de Matem\'atica, `Alberto P. Calder\'on', CONICET, Saavedra 15 3er. piso,
(1083) Buenos Aires, Argentina.}     
                                               
\noi e-mail: {\sf eduardo@mate.unlp.edu.ar}

\bigskip

{\sc (Gabriel Larotonda)} {Instituto de Ciencias,  Universidad Nacional de Gral. Sar\-miento,
J.M. Gutierrez 1150,  (1613) Los Polvorines, Argentina and Instituto Argentino de Matem\'atica, `Alberto P. Calder\'on', CONICET, Saavedra 15 3er. piso,
(1083) Buenos Aires, Argentina.}

\noi e-mail: {\sf glaroton@ungs.edu.ar }

\end{document}